\documentclass[12pt]{amsart}
\usepackage{hyperref}
\usepackage{graphicx, labelfig}
\usepackage{pstricks}

\theoremstyle{plain}
\newtheorem{theorem}{Theorem}[section]
\newtheorem{proposition}[theorem]{Proposition}
\newtheorem{lemma}[theorem]{Lemma}

\theoremstyle{definition}

\newcommand{\R}{\mathbb{R}}

\newcommand{\RP}{\mathbb{RP}^2}

\newcommand{\PGL}{\mathrm{PGL}_3(\R)}
\newcommand{\SL}{\mathrm{SL}_3(\R)}

\title[The Goldman and Fock-Goncharov coordinates]
 {The Goldman and Fock-Goncharov coordinates\\
 for convex projective structures on surfaces}
\author{Francis Bonahon}
\address{Department of Mathematics, University of Southern California, Los Angeles CA 90089-2532, U.S.A.}
\email{fbonahon@math.usc.edu}
\author{Inkang Kim}
\address{School of Mathematics, Korean Institute for Advanced Study, Heogiro 85, Dongdaemen-gu, Seoul, 130-722, Korea}
\email{inkang@kias.re.kr}


\thanks{This research is based upon work supported by the U.S. National Science Foundation under Grant 0932078000, while both authors were in residence at the Mathematical Sciences Research Institute in Berkeley, California, in the Spring 2015 semester. It was also partially supported by NSF Grants DMS-1105402 and DMS-1406559, Grant NRF-2014R1A2A2A01005574 from the National Research Foundation of Korea, as well as by a Fellowship from the Simons Foundation (Grant 301050).}

\dedicatory{Dedicated to Bill Goldman,\\
on the occasion of his 60-th birthday}

\date{\today}

\begin{document}

\begin{abstract} Let $\mathfrak P(S)$ be the space of convex projective structures on a surface $S$ with negative Euler characteristic. Goldman and Bonahon-Dreyer constructed two different sets of global coordinates for $\mathfrak P(S)$, both associated to a pair of pants decomposition of the surface~$S$. The article explicitly describes the coordinate change between these two parametrizations.  Most of the arguments are concentrated in the case where $S$ is a pair of pants, in which case the Bonahon-Dreyer coordinates are actually due to Fock-Goncharov. 
\end{abstract}

\maketitle

Let $S$ be a compact surface, possibly with boundary. A \emph{projective structure} on this surface locally models $S$ over the real projective plane $\RP$, in such a way that the boundary  locally corresponds to a straight line in $\RP$ and that coordinate changes are induced by projective transformations of $\RP$.  Recall that the group of projective transformations of $\RP$ is $\PGL=\SL$. A projective structure $\pi$ on $S$ lifts to a projective structure on the universal cover $\widetilde S$, for which there exists a  projective map $\mathrm{dev}_\pi \colon \widetilde S \to \RP$ which is equivariant with respect to a group homomorphism $\rho_\pi \colon \pi_1(S) \to \SL$. This \emph{developing map} $\mathrm{dev}_\pi$ is unique up to composition with the projective map $\RP \to \RP$ induced by an element  $A\in \SL$; its \emph{monodromy homomorphism} $\rho_\pi$ is then determined up to conjugation by the same $A\in \SL$. The projective structure is \emph{convex} if the developing map $\mathrm{dev}_\pi$ induces a homeomorphism between $\widetilde S$ and a convex domain $\Omega$ in $\RP$.  

The modern theory of convex projective structures on surfaces received a great boost from two very influential articles of Bill Goldman  \cite{Goldman, ChoiGold}, the second one in collaboration with Suhyoung Choi.
%Choi and Goldman \cite{ChoiGold} show that projective transformations are completely determined by their monodromies, and that the monodromies of convex projective structures form a whole component in the space of all homomorphisms $\pi_1(S) \to \SL$, considered up to conjugation by elements of $\SL$,  that send each boundary component to a diagonalizable matrix with distinct eigenvalues.  
In particular, when the surface $S$ has negative Euler characteristic $\chi(S)$, Goldman \cite{Goldman} considers the space $\mathfrak P(S)$ of isotopy classes of convex projective structures on $S$, and shows that $\mathfrak P(S)$ is diffeomorphic to an open cell of dimension $8|\chi(S)|$. He proves this result by constructing rather explicit coordinates for $\mathfrak P(S)$, associated to a pair of pants decomposition of the surface $S$.

The purpose of the current paper is to compare Goldman's parametrization of $\mathfrak P(S)$ to another parametrization more recently developed by Dreyer and the first author in \cite{BonDre1}. More precisely, we give an explicit correspondence between Goldman's coordinates for $\mathfrak P(S)$ and the coordinates of \cite{BonDre1}, when these two sets of coordinates are associated to the same pair of pants decomposition of the surface $S$. The existence of such explicit coordinate changes is of course not surprising, and similar computations can be found in \cite[\S4.7]{AleChoi}.  However the authors thought that it would be useful to have them available in print. See for instance \cite{Zhang1, Zhang2} for recent work that uses the two points of view. See also \cite[\S5]{Lawton} for a correspondence, in the case of the pair of pants,  between the Goldman coordinates and the trace coordinates developed in  \cite{Lawton}.

\section{The Goldman parameters for the pair of pants}
\label{sect:Goldman}

We first consider the elementary blocks of Goldman's parametrization, namely the case where $S$ is  a pair of pants, with boundary components $A_1$, $A_2$, $A_3$. These boundary components are called $A$, $B$, $C$ in \cite{Goldman}, but our convention seems a little more reader-friendly for the subsequent computations. For compatibility with \cite{Goldman}, we orient the $A_i$ by the opposite of the boundary orientation induced by the orientation of $S$.

Goldman associates eight positive real parameters $\lambda_1$, $\tau_1$, $\lambda_2$, $\tau_2$, $\lambda_3$, $\tau_3$, $s$, $t>0$ to each convex projective structure on the pair of pants $S$, and shows that the resulting map $\mathfrak P(S) \to \R^8$ induces a diffeomorphism between $\mathfrak P(S)$ and the open subset $U\subset \R^8$ defined by  the inequalities 
$$2\lambda_i^{-\frac12} < \tau_i < \lambda_i + \lambda_i^{-2}$$ 
for every $i=1$, $2$, $3$. 
We now describe these parameters. 

The parameters $\lambda_i$ and $\tau_i$ are associated to the $i$--th boundary component $A_i$ of the pair of pants  $S$, and more precisely to the monodromy $\rho_\pi(A_i)\in \SL$ defined by the projective structure $\pi \in \mathfrak P(S)$ considered. Goldman shows that the eigenvalues $0<\lambda_i < \mu_i < \nu_i$ of $\rho_\pi(A_i)\in \SL$ are positive real and distinct. The invariant $\lambda_i$ is then the smallest one of these eigenvalues, while $\tau_i = \mu_i + \nu_i$ is the sum of the other two. 

The construction of the remaining parameters $s$ and $t$, called the \emph{internal parameters} of the projective structure $\pi\in \mathfrak P(S)$,  is much more elaborate.

\begin{figure}[htbp]
\vskip 5pt
\SetLabels
( .19* .35) $\Sigma $ \\
( .2* 1.05) $ p_2$ \\
(.4 *-.05 ) $p_3 $ \\
(-.02 * -.05) $p_1 $ \\
( .2* -.12) $ B_2$ \\
( .05* .5) $B_3 $ \\
( .32* .5) $B_1 $ \\
(.63 *.7 ) $T_- $ \\
( .8* .35) $T_+$ \\
( .8* -.12) $ B_2$ \\
( .67*.5 ) $B_3 $ \\
( .49* .5) $B_2 $ \\
(.63 * 1.03) $ B_1$ \\
( .94* .5) $B_1 $ \\
\endSetLabels
\centerline{\AffixLabels{\includegraphics{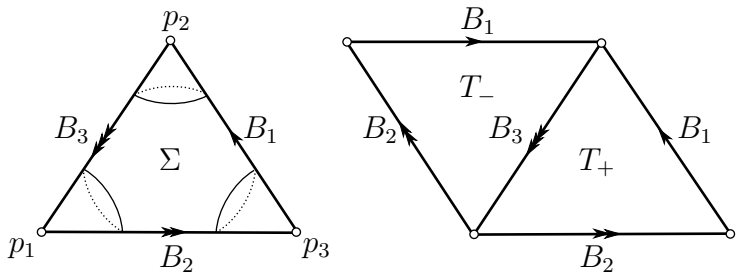}}}
\vskip 5pt
\caption{The three-puncture sphere $\Sigma$}
\label{fig:Pants}
\end{figure}

Identify the interior $S-\partial S$ to a three-puncture sphere $\Sigma = S^2 - \{p_1, p_2, p_3\}$ as in Figure~\ref{fig:Pants}, and decompose $\Sigma$ into two ideal triangles $T_+$ and $T_-$ meeting along disjoint lines $B_1$, $B_2$ and $B_3$ going from puncture to puncture. Considering indices modulo 3, we choose the indexing so that the puncture $p_i$ of $S-\partial S$ corresponds to the component $A_i$ of $\partial S$, the line $B_i$ goes from the puncture $p_{i-1}$ to $p_{i+1}$, and the punctures $p_1$, $p_2$, $p_3$ occur in this order as one goes clockwise around the boundary of $T_+$ (and counterclockwise around the boundary of $T_-$). See Figure~\ref{fig:Pants}. The lines $B_1$, $B_2$, $B_3$ are called $a$, $b$, $c$ in \cite{Goldman}. 

The triangles $T_+$ and $T_-$ lift to a family of ideal triangles that tessellate the universal cover $\widetilde \Sigma \subset \widetilde S$ of the punctured sphere $\Sigma \subset S$. Given a projective structure $\pi \in \mathfrak P(S)$, Goldman isotops this projective structure on $\Sigma$ so that the developing map $\mathrm{dev} \colon \widetilde \Sigma \to \RP$ sends each lift $\widetilde T \subset \widetilde\Sigma$ of $T_{\pm}$ to a geometric triangle $\Delta$ (delimited by three straight line segments) in $\RP$, minus the vertices of $\Delta$. By construction, each vertex $v$ of such a triangle $\Delta= \mathrm{dev}(\widetilde T) \subset \RP$ is invariant under the action  $\rho_\pi(A_i) \in \SL$ for some element $A_i\in \pi_1(S)$ of the conjugacy class determined by  a component $A_i $ of the boundary $\partial S$. Goldman arranges in addition that this  vertex $v\in \RP$ is the repelling fixed point of $\rho_\pi(A_i)$, corresponding to the eigenspace associated to the smallest eigenvalue $\lambda_i$ of $\rho_\pi(A_i)$.

In the universal cover $\widetilde \Sigma \subset \widetilde S$, choose a component $\widetilde T_+$ of the preimage  of the ideal triangle $T_+$, and let $\widetilde T_1$, $\widetilde T_2$, $\widetilde T_3$ be the components of the preimage of $T_-$ that touch $\widetilde T_+$ along the sides that  correspond to $B_1$, $B_2$, $B_3$, respectively. Let $\Delta_+ = \mathrm{dev}(\widetilde T_+)$, $\Delta_1 = \mathrm{dev}(\widetilde T_1)$, $\Delta_2 = \mathrm{dev}(\widetilde T_2)$, $\Delta_3 = \mathrm{dev}(\widetilde T_3)$ be the corresponding triangles in $\RP$.  We already observed that, for the monodromy $\rho_\pi \colon \pi_1(S) \to \SL$, each vertex of $\Delta_+$ is the repelling fixed point of some $\rho_\pi(A_i) \in \SL$ for some element $A_i$ represented  by a  component $A_i$ of $\partial S$. Looking at the action of these fundamental group elements $A_i \in \pi_1(S)$ on $\widetilde S$ and on the triangles $\widetilde T$ of the preimage of $T_\pm$, we see that each $\rho_\pi(A_i) \in \SL$ sends the triangle $\Delta_{i+1}$ to $\Delta_{i-1}$, considering indices modulo 3. See Figure~\ref{fig:Triangles1}.

\begin{figure}[htbp]
\vskip 10pt
\SetLabels
( .49* .82) $\rho_\pi( A_2)$ \\
(.76 * .35) $\rho_\pi(A_3) $ \\
( .21* .37) $\rho_\pi( A_1)$ \\
(.53 * .16) $ \Delta_2$\\
( .16*.67 ) $\Delta_3 $\\
( .84*.71 ) $ \Delta_1$\\
( .5* .54) $ \Delta_+$\\
( .5*1.02 ) $[ 0,1,0] $ \\
( 1.07* .27) $[0,0,1 ] $\\
(-.1 * .27) $[1,0,0] $ \\
(.54 * -.05) $[a_2 ,-1 ,c_2 ] $ \\
( -.11* .72) $ [ a_3, b_3, -1]$ \\
(1.15 * .82) $[ -1,b_1 ,c_1 ] $ \\
\endSetLabels
\centerline{\AffixLabels{\includegraphics{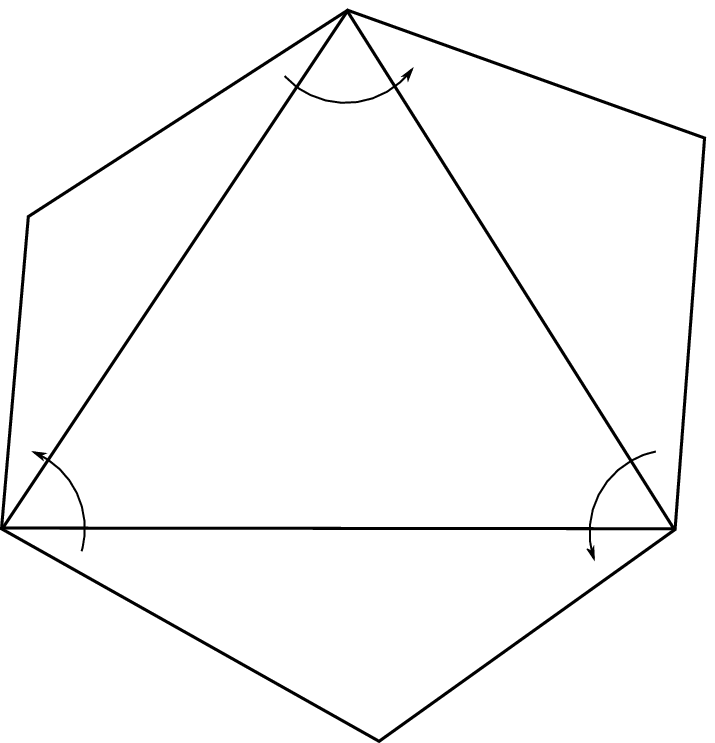}}}
\vskip 10pt
\caption{}
\label{fig:Triangles1}
\end{figure}

The developing map $\mathrm{dev}\colon \widetilde \Sigma \to \RP$ is only defined up to postcomposition with an element of $\SL$. We can therefore arrange that the vertices of the triangle $\Delta$ that are respectively fixed by $\rho_\pi(A_1)$, $\rho_\pi(A_2)$ and $\rho_\pi(A_3)$ are the points of $\RP$ with homogeneous coordinates $[1,0,0]$, $[0,1,0]$ and $[0,0,1]$, respectively. Goldman shows that, as a consequence of the convexity of the projective structure, the remaining vertices of the triangles $\Delta_1$, $\Delta_2$, $\Delta_3$ can be written as $[-1, b_1, c_1]$, $[a_2, -1, c_2]$, $[a_3, b_3, -1]$ with all $a_i$, $b_i$, $c_i>0$.

The internal parameters $s$, $t>0$ of the convex projective structure $\pi \in \mathfrak P(S)$ are then defined by the property that
$$
t= \frac{a_2b_3}{a_3}
$$
and
\begin{align*}
b_3c_2 &= 1 +  \tau_1 \sqrt{\frac{\lambda_1\lambda_3}{\lambda_2}} s + \frac{\lambda_3}{\lambda_2}s^2,\\
a_3c_1 &= 1 +  \tau_2 \sqrt{\frac{\lambda_1\lambda_2}{\lambda_3}} s + \frac{\lambda_1}{\lambda_3}s^2,\\
a_2b_1 &= 1 +  \tau_3 \sqrt{\frac{\lambda_2\lambda_3}{\lambda_1}} s + \frac{\lambda_2}{\lambda_1}s^2.
\end{align*}
The numbers $\rho_1=b_3c_2$, $\rho_2=a_3c_1$ and $\rho_3=a_2b_1$ occurring here are the crossratios of the four lines passing through each vertex of $\Delta_+$ in Figure~\ref{fig:Triangles1}, and Goldman shows that these crossratios are all greater than 1. As a consequence, $s$ is determined as the unique positive solution to any one of the three equations above.

Note that  the existence of $s$ imposes constraints between the coordinates $a_i$, $b_i$, $c_i$ and  the eigenvalue invariants $\lambda_i$ and $\tau_i$.

The definition of the parameter $s$ is rather intrinsic and symmetric. For instance, the projective structure on $S$ comes from a hyperbolic metric if and only if $s=1$ and $\tau_i =1+\lambda_i^{-1}$ for each $i=1$, $2$, $3$.   The construction of $t$ involves a symmetry break, and this parameter depends on the indexing of the boundary components of $S$.

\section{The Fock-Goncharov coordinates for the pair of pants}
\label{sect:FockGoncharov}

We now turn to the coordinates of \cite{BonDre1}, still for the pair of pants $S$. In this case, these coordinates are actually due to Fock and Goncharov \cite{FG, FG2}. They consist of two \emph{triangle invariants} $\tau_{111}(T_+)$ and $\tau_{111}(T_-)$ associated to the triangle $T_+$ and $T_-$, and two \emph{shear invariants} $\sigma_1(B_i)$ and $\sigma_2(B_i)$ associated to each of the oriented lines $B_1$, $B_2$, $B_3$. 

The Fock-Goncharov coordinates for a convex projective structure $\pi\in\mathfrak P(S)$ require that we choose, for each boundary component $A_i$ of $S$, a flag $F_i$ which is invariant under the monodromy $\rho_\pi(A_i) \in \SL$ of $\pi$. Recall that a \emph{flag} in $\R^3$ is a family $F$ of linear subspaces $0= F^{(0)}\subset  F^{(1)} \subset F^{(2)} \subset F^{(3)}=\R^3$ where each $F^{(a)}$ has dimension $a$. To be more precise since $A_i\in \pi_1(S)$ is only defined up to conjugation, we need to choose an invariant flag for the image under $\rho_\pi\colon \pi_1(S) \to \SL$ of each element of the corresponding conjugacy class $\pi_1(S)$, in a $\rho_\pi$--equivariant way with respect to the operation of conjugation by elements of $\pi_1(S)$. 

In view of Goldman's conventions, it is natural to choose for $F_i$ the \emph{unstable flag} of $\rho_\pi(A_i)$, whose line $F_i^{(1)}$ is the eigenspace of $\rho_\pi(A_i)$ corresponding to the lowest eigenvalue $\lambda_i$, and whose plane $F_i^{(2)}$ is generated by $F_i^{(1)}$ and by the eigenspace corresponding to the second lowest eigenvalue $\mu_i$ of $\rho_\pi(A_i)$. (Recall that $\rho_\pi(A_i)\in \SL$ is represented by a matrix of $\SL$ with distinct positive real eigenvalues $0 < \lambda_i < \mu_i < \nu_i$.)

In the geodesic lamination setup of \cite{BonDre1}, this means that we are considering the lines $B_1$, $B_2$, $B_3$ as spiraling to the left along the boundary components of $S$, as in Figure~\ref{fig:PantsSpiral}.

\begin{figure}[htbp]

\SetLabels
( .43* .45) $ T_+$ \\
( .7*.55 ) $T_-$ \\
(.43 *.29 ) $ B_2$ \\
( .28* .47) $ B_3$ \\
( .53* .58) $ B_1$ \\
( .5* .92) $ A_2$ \\
( .17* -.07) $ A_1$ \\
( .85*-.07 ) $ A_3$ \\
\endSetLabels
\centerline{\AffixLabels{ \includegraphics{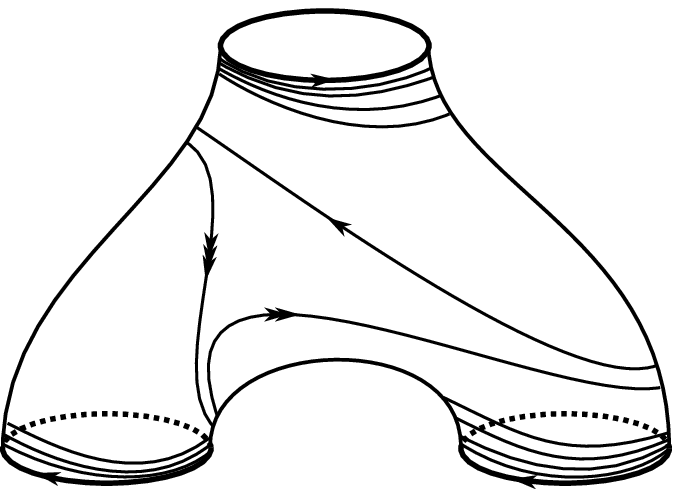}}}

\caption{}
\label{fig:PantsSpiral}
\end{figure}

Lift $T_+$ to a triangle $\widetilde T_+$ in the universal cover $\widetilde S$, and represent the boundary components of $S$ by homotopy classes $A_1$, $A_2$, $A_3\in \pi_1(S)$ that each fix one vertex of $\widetilde T_+$. For the orientation of $S$, the vertices of  $\widetilde T_+$ that are respectively fixed by $A_1$, $A_2$, $A_3\in \pi_1(S)$ occur clockwise in this order around $\widetilde T_+$. Then, if we follow the sign conventions of \cite{BonDre1} (which are the opposite of those of \cite{BonDre2}), the definition of the triangle invariant $\tau_{111}(T_+)$ is
$$
\tau_{111}(T_+) = \log
 \frac{f_1^{(2)}\wedge f_2^{(1)}}{f_2^{(1)}\wedge f_3^{(2)}}  
\frac{f_1^{(1)}\wedge f_3^{(2)}}{f_1^{(1)}\wedge f_2^{(2)}}  \frac{f_2^{(2)}\wedge f_3^{(1)}}{f_1^{(2)}\wedge f_3^{(1)}}
$$
for arbitrary non-zero elements $f_i^{(a)} \in \Lambda^a(F_i^{(a)})$, where $F_i \in \mathrm{Flag}(\R^3)$ is the unstable flag of the element $A_i \in \pi_1(S)$. Fock and Goncharov prove that this definition makes sense, as the triple ratio appearing inside of the $\log$ is positive. 

Similarly, lift $T_-$ to a triangle $\widetilde T_- \subset \widetilde S$ whose vertices are fixed by  classes $A_1'$, $A_2'$, $A_3'\in \pi_1(S)$, and let $F_i' \in \mathrm{Flag}(\R^3)$ be the unstable flag of $A_i' \in \pi_1(S)$. Then, because the vertices  respectively fixed by $A_1'$, $A_2'$, $A_3'\in \pi_1(S)$ now occur counterclockwise in this order around $\widetilde T_-$,
$$
\tau_{111}(T_-) = \log
 \frac{f_1^{\prime(2)}\wedge f_3^{\prime(1)}}{f_2^{\prime(2)}\wedge f_3^{\prime(1)}}  
\frac{f_1^{\prime(1)}\wedge f_2^{\prime(2)}}{f_1^{\prime(1)}\wedge f_3^{\prime(2)}}  \frac{f_2^{\prime(1)}\wedge f_3^{\prime(2)}}{f_1^{\prime(2)}\wedge f_2^{\prime(1)}}
$$
for an arbitrary choice of non-zero elements $f_i^{\prime(a)} \in \Lambda^a(F_i^{\prime(a)})$.

We now define the shearing invariants $\sigma_1(B_i)$ and $\sigma_2(B_i)$ along each of the three spiraling leaves $B_1$, $B_2$, $B_3$. For this, we orient $B_i$ so that it goes from $A_{i-1}$ to $A_{i+1}$, as in Figures~\ref{fig:Pants} and \ref{fig:PantsSpiral}; this orientation also coincides with the boundary orientation of the triangle $T_+$. Lift $B_i$ to a line $\widetilde B_i$ in the universal covering $\widetilde S$, and let $\widetilde T_+$ and $\widetilde T_-$ be the lifts of $T_+$ and $T_-$ that are adjacent to $\widetilde B_i$. Let $F_{i+1}$, $F_{i-1}$, $F_i$ and $F_i'\in \mathrm{Flag}(\R^3)$ be the unstable flags respectively associated to the positive endpoint of $\widetilde B_i$,  the negative endpoint of $\widetilde B_i$, the third vertex of $\widetilde T_+$, and the third vertex of $\widetilde T_-$. Then
$$
\sigma_1(B_i) = \log \left(-
\frac{f_{i+1}^{(1)} \wedge f_{i-1}^{(1)} \wedge f_i^{(1)}}
{f_{i+1}^{(1)} \wedge f_{i-1}^{(1)} \wedge f_i^{\prime(1)}}
\frac{ f_{i-1}^{(2)} \wedge f_i^{\prime(1)}}
{f_{i-1}^{(2)} \wedge f_i^{(1)}}
\right)
$$
and
$$
\sigma_2(B_i) = \log \left(-
\frac{f_{i+1}^{(1)} \wedge f_{i-1}^{(1)} \wedge f_i^{\prime(1)}}
{f_{i+1}^{(1)} \wedge f_{i-1}^{(1)} \wedge f_i^{(1)}}
\frac{ f_{i+1}^{(2)} \wedge f_i^{(1)}}
{f_{i+1}^{(2)} \wedge f_i^{\prime(1)}}
\right)
$$
with our usual conventions that $f_j^{(a)} \in \Lambda^a(F_j^{(a)})$ and $f_j^{\prime(a)} \in \Lambda^a(F_j^{\prime(a)})$. Again, this definition makes sense as Fock and Goncharov show that the quantities inside of the $\log$ are positive.

\section{The pair of pants: from the Fock-Goncharov coordinates to the Goldman coordinates}
\label{sect:PantsFromFGtoGoldman}

We now indicate, for the pair of pants $S$,  how to compute the Goldman coordinates of a projective structure $\pi\in \mathfrak P(S)$ from its Fock-Goncharov coordinates. 

We begin with the boundary invariants $\lambda_1$, $\tau_1$, $\lambda_2$, $\tau_2$, $\lambda_3$, $\tau_3$ associated to the boundary components $A_1$, $A_2$, $A_3$ of $S$.

\begin{proposition}
\label{prop:GoldmanBoudaryFromFockGoncharov}
For $i=1$, $2$, $3$, 
\begin{align*}
\lambda_i &=\mathrm e^{  { \frac13} \sigma_1(B_{i+1}) +  { \frac23} \sigma_2(B_{i+1}) + { \frac23} \sigma_1(B_{i-1}) + { \frac13} \sigma_2(B_{i-1}) +{ \frac23}\tau_{111}(T_+) + { \frac23}\tau_{111}(T_-) }
\\
\text{and }\tau_i &= \bigl( \mathrm e^{ -\sigma_1(B_{i+1}) -\sigma_2(B_{i-1})}+1  \bigr) \mathrm e^{  { \frac13} \sigma_1(B_{i+1}) -  { \frac13}\sigma_2(B_{i+1}) - { \frac13} \sigma_1(B_{i-1}) +{ \frac13}\sigma_2(B_{i-1}) -{ \frac13}\tau_{111}(T_+) -{ \frac13}\tau_{111}(T_-)  }.
\end{align*}
\end{proposition}
\begin{proof}
This is an immediate consequence of Proposition~13 of \cite{BonDre1} which, if $\rho_\pi(A_i) \in \SL$ is represented by a matrix of $\SL$ with eigenvalues $0 < \lambda_i < \mu_i < \nu_i$, determines $\ell_1(A_i) = \log \nu_i - \log \mu_i$ and $\ell_2(A_i) = \log \mu_i -\log \lambda_i$ in terms of the Fock-Goncharov invariants $\tau_{111}(T_{\pm})$, $\sigma_1(B_j)$, $\sigma_2(B_j)$. More precisely,
\begin{align*}
\ell_1(A_i) &= - \sigma_1(B_{i+1}) - \sigma_2(B_{i-1}) ,
\\
\ell_2(A_i) &= - \sigma_2(B_{i+1}) - \sigma_1(B_{i-1}) -\tau_{111}(T_+) -\tau_{111}(T_-) .
\end{align*}
Since $\lambda_i \mu_i \nu_i=1$, this gives
\begin{align*}
\log \lambda_i & = {\textstyle \frac13} \sigma_1(B_{i+1}) +  {\textstyle \frac23} \sigma_2(B_{i+1}) + {\textstyle \frac23} \sigma_1(B_{i-1}) + {\textstyle \frac13} \sigma_2(B_{i-1}) +{\textstyle \frac23}\tau_{111}(T_+) + {\textstyle \frac23}\tau_{111}(T_-),
\\
\log \mu_i &=  {\textstyle \frac13} \sigma_1(B_{i+1}) -  {\textstyle \frac13}\sigma_2(B_{i+1}) - {\textstyle \frac13} \sigma_1(B_{i-1}) +  {\textstyle \frac13}\sigma_2(B_{i-1}) -{\textstyle \frac13}\tau_{111}(T_+) -{\textstyle \frac13}\tau_{111}(T_-) ,
\\
\log\nu_i &=  -{\textstyle \frac23}\sigma_1(B_{i+1}) -  {\textstyle \frac13}\sigma_2(B_{i+1}) -  {\textstyle \frac13}\sigma_1(B_{i-1}) -{\textstyle \frac23} \sigma_2(B_{i-1}) -{\textstyle \frac13}\tau_{111}(T_+) -{\textstyle \frac13}\tau_{111}(T_-)  .
\end{align*}
The computation is then completed by the property that $\tau_i = \mu_i + \nu_i$. 
\end{proof}

To compute the internal invariants $s$ and $t$, we can modify the configuration of the triangles $\Delta_+$, $\Delta_1$, $\Delta_2$, $\Delta_3$ by a projective transformation so that we still have $F_1^{(1)}=[1,0,0]$,  $F_2^{(1)}=[0,1,0]$, $F_3^{(1)}=[0,0,1]$ as in Figure~\ref{fig:Triangles1}, but  so that in addition the intersection line of the planes $F_1^{(2)} \cap F_3^{(2)}$ corresponds to the point $[1, -1, 1]\in \RP$. Then, the intersections $F_1^{(2)} \cap F_2^{(2)}$ and $F_2^{(2)} \cap F_3^{(2)}$ respectively correspond to points $[x,1,-1]$ and $[-x, x, 1]\in \RP$ for some $x>0$. See Figure~\ref{fig:Triangle3}. 

%\begin{figure}[htbp]

%\SetLabels
%( .45* .3) $\Delta_+'$ \\
%( .92* .42) $\Delta_1' $ \\
%( .5*.1 ) $ \Delta_2'$ \\
%( .09* .42) $ \Delta_3'$ \\
%(-.1 * .18) $F_1^{(1)}=[1,0,0] $ \\
%( 1.07* .18) $F_3^{(1)}=[0,0,1] $ \\
%(.15 * .75) $F_2^{(1)}=[1,x+1,x] $ \\
%(1.05 * .52) $F_1^{\prime(1)}=[1,b_1', c_1'] $ \\
%(-.01 *.52 ) $F_3^{\prime(1)}=[a_3', b_3', 1] $ \\
%( .5* -.05) $F_2^{\prime(1)}=[a_2', -1, c_2'] $ \\
%( .85*.6 ) $ F_2^{(2)}$ \\
%(.61 *.97 ) $F_1^{(2)} $ \\
%(.38 * .98) $F_3^{(2)} $ \\
%(.74 *.86 ) $ F_1^{(2)}\cap F_3^{(2)} = [0,1,0]$ \\
%( .95* .32) $F_1^{\prime(2)} $ \\
%( .97*.73 ) $ F_2^{(2)} \cap F_1^{\prime(2)}=[1,y+1,y] $ \\
%\endSetLabels
%\centerline{\AffixLabels{\includegraphics{Triangles2.eps}}}
%\vskip 10pt
%\caption{}
%\label{fig:Triangles2}
%\end{figure}

\begin{figure}[htbp]

\SetLabels
( .42*.78) $F_2^{(1)}=[ 0,1,0] $ \\
( .75* .43) $F_3^{(1)}=[0,0,1 ] $\\
(.18 * .42) $F_1^{(1)}=[1,0,0] $ \\
(.1*.85) $F_1^{(2)}$\\
(.71*.85) $F_2^{(2)}$\\
(.79*.63) $F_3^{(2)}$\\
(.55 * .12) $[1, -1, 1]$\\
(.1 * .64) $[x, 1, -1]$\\
(.98 * .88) $[-x, x, 1]$\\
\tiny
(.45 * .32) $[a_2 ,-1 ,c_2 ] $ \\
( .12* .54) $ [ a_3, b_3, -1]$ \\
(.68 * .7) $[ -1,b_1 ,c_1 ] $ \\
\endSetLabels

\centerline{\AffixLabels{\includegraphics{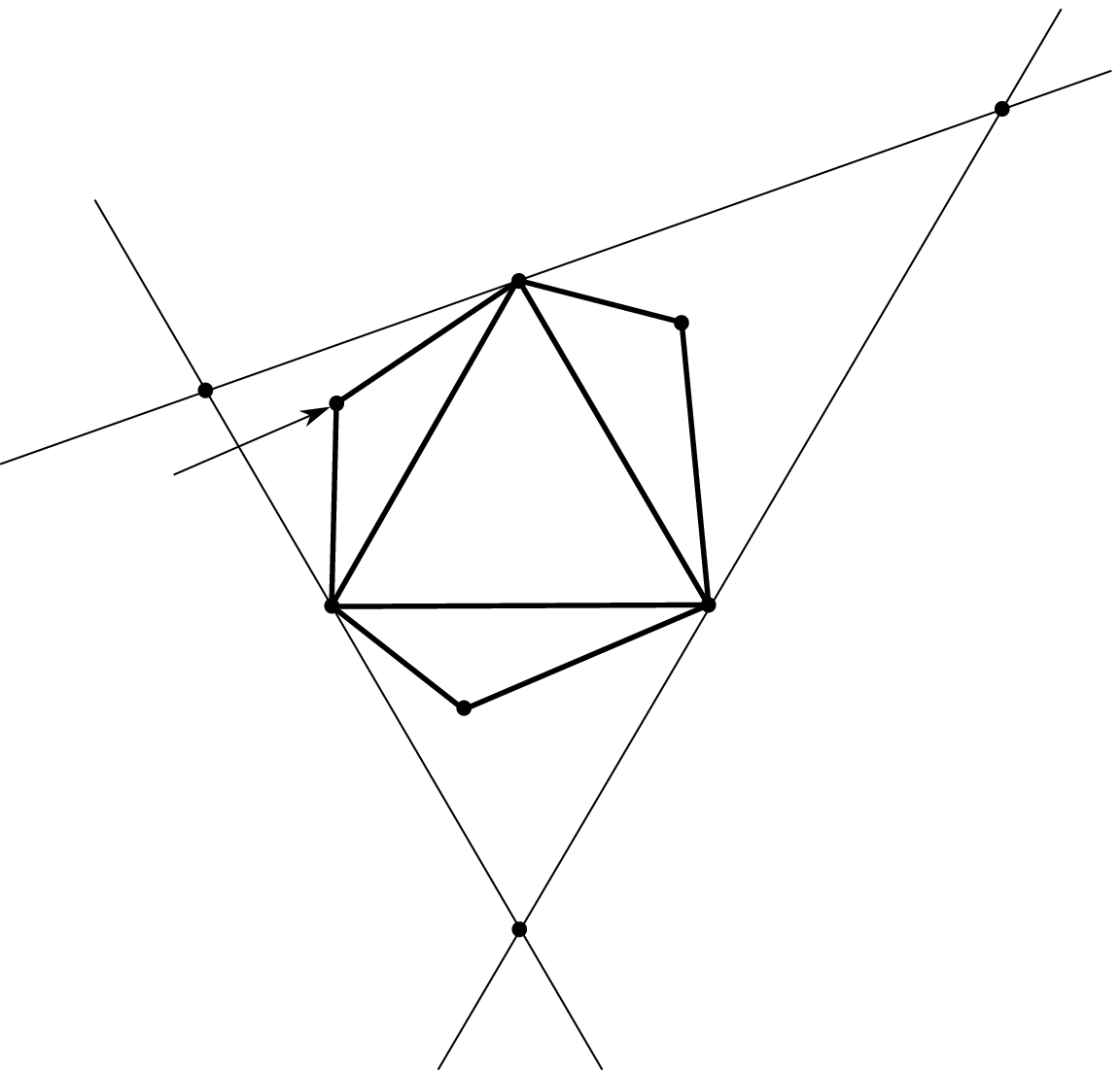}}}

\caption{}
\label{fig:Triangle3}
\end{figure}

With this normalization and using $F_1^{\prime(1)}=[-1,b_1, c_1]$,  $F_2^{\prime(1)}=[a_2,-1,c_2]$, $F_3^{\prime(1)}=[a_3, b_3,-1]$, the Fock-Goncharov invariants are now
\begin{align*}
\sigma_1(B_1) &= \log(b_1-1) &
\sigma_2(B_1) &= -\log(xc_1-1)\\
\sigma_1(B_2) &= \log(c_2-1) &
\sigma_2(B_2) &= -\log(a_2-1)\\
\sigma_1(B_3) &= \log(a_3x^{-1}-1) &
\sigma_2(B_3) &= -\log(b_3-1)\\
\tau_{111}(T_+)  &=\log x&&\\
\end{align*}

The triangle invariant $ \tau_{111}(T_-)$ is harder to compute from this data, but we will not need it at this point. 

Solving the above equations gives that 
\begin{align*}
b_1 &= \mathrm e^{\sigma_1(B_1)} +1 &
c_1 &=  \mathrm e^{-\tau_{111}(T_+)} (\mathrm e^{-\sigma_2(B_1)}+1)\\
a_2&= \mathrm e^{-\sigma_2(B_2)} +1 &
c_2 &= \mathrm e^{\sigma_1(B_2)}+1\\
a_3 &= \mathrm e^{\tau_{111}(T_+)}(\mathrm e^{\sigma_1(B_3)} +1) &
b_3 &= \mathrm e^{-\sigma_2(B_3)}+1
\end{align*}

In particular, this enables us to compute the invariant
$$
t=\frac{a_2 b_3}{a_3} = \mathrm e^{-\tau_{111}(T_+)}(\mathrm e^{-\sigma_2(B_2)} +1)( \mathrm e^{-\sigma_2(B_3)}+1)(\mathrm e^{\sigma_1(B_3)} +1)^{-1}
$$
 in terms of the Fock-Goncharov invariants $\sigma_i(B_j)$ and $\tau_{111}(T_+)$. Similarly, since Goldman's boundary invariants $\lambda_i$ and $\tau_j$ are determined by Proposition~\ref{prop:GoldmanBoudaryFromFockGoncharov}, the internal invariant $s$ is determined as the positive solution to any of the following three equations
 \begin{align*}
( \mathrm e^{\sigma_1(B_2)}+1)(\mathrm e^{-\sigma_2(B_3)}+1) &= 1 +  \tau_1 \sqrt{\frac{\lambda_1\lambda_3}{\lambda_2}} s + \frac{\lambda_3}{\lambda_2}s^2,\\
 (\mathrm e^{\sigma_1(B_3)} +1) (\mathrm e^{-\sigma_2(B_1)}+1) &= 1 +  \tau_2 \sqrt{\frac{\lambda_1\lambda_2}{\lambda_3}} s + \frac{\lambda_1}{\lambda_3}s^2,\\
( \mathrm e^{\sigma_1(B_1)} +1) (\mathrm e^{-\sigma_2(B_2)} +1 )&= 1 +  \tau_3 \sqrt{\frac{\lambda_2\lambda_3}{\lambda_1}} s + \frac{\lambda_2}{\lambda_1}s^2.
\end{align*}

	After substituting into any of the above equations the expressions for the $\lambda_i$ and $\tau_i$ given  by Proposition~\ref{prop:GoldmanBoudaryFromFockGoncharov} and applying the Quadratic Formula, a long but elementary computation then leads to the following remarkably simple expression for $s$. (This long computation could be simplified by using a different expression of the parameter $s$, also found in \cite{Goldman}, but at the expense of making our \S \ref{sect:Goldman} more complicated.)
	
\begin{proposition}
\label{prop:InternalParam}
In the pair of pants, the  internal parameters $s$ and $t$ are respectively equal to
\begin{align*}
s&=  \mathrm e^{\frac16\big( \sigma_1(B_1) +\sigma_1(B_2) +\sigma_1(B_3) -\sigma_2(B_1) -\sigma_2(B_2) -\sigma_2(B_3) \big)} 
\\
\text{and } t&=  \mathrm e^{-\tau_{111}(T_+)}(\mathrm e^{-\sigma_2(B_2)} +1)( \mathrm e^{-\sigma_2(B_3)}+1)(\mathrm e^{\sigma_1(B_3)} +1)^{-1}.
\end{align*}
\vskip -\belowdisplayskip
\vskip -\baselineskip \qed
\end{proposition}

\vskip \belowdisplayskip

\section{The pair of pants: from the Goldman coordinates to the Fock-Goncharov coordinates}
\label{sect:PantsFromGoldmanToFG}

It is  elementary to invert the formulas of Propositions~\ref{prop:GoldmanBoudaryFromFockGoncharov} and \ref{prop:InternalParam}, in order to express the Fock-Goncharov coordinates in terms of the Goldman coordinates.

\begin{proposition}
\label{prop:FromGoldmanToFockGoncharov}
The Fock-Goncharov shear coordinates $\sigma_1(B_1)$, $\sigma_1(B_2)$, $\sigma_1(B_3)$, $\sigma_2(B_1)$, $\sigma_2(B_2)$, $\sigma_2(B_3)$ are expressed in terms of the Goldman coordinates $\lambda_1$, $\lambda_2$, $\lambda_3$, $\tau_1$, $\tau_2$, $\tau_3$, $s$, $t$ by the property that
\begin{align*}
\sigma_1(B_i)&= \log \left( s{\mu_{i-1}}  \sqrt{\frac{\lambda_{i-1}\lambda_{i+1}}{\lambda_i}} \right) 
\\
\sigma_2(B_i)&= 
\log \left( \frac{\mu_{i+1}}s  \sqrt{\frac{\lambda_{i-1}\lambda_{i+1}}{\lambda_i}} \right) 
\end{align*}
where
$$ 
\mu_i = \frac{\tau_i  - \sqrt{\tau_i^2 - \frac4{ \lambda_i}}}{2}.
$$
Also, the triangle invariants $\tau_{111}(T_+)$ and $\tau_{111}(T_+)$ are given by
\begin{align*}
\tau_{111}(T_+) &= \log \frac{(\mathrm e^{-\sigma_2(B_2)} +1)( \mathrm e^{-\sigma_2(B_3)}+1)}{t(\mathrm e^{\sigma_1(B_3)} +1)}
\\
\tau_{111}(T_-) &= \log \frac {t\mu_1\mu_2\mu_3 (\mathrm e^{\sigma_1(B_3)} +1)} {(\mathrm e^{-\sigma_2(B_2)} +1)( \mathrm e^{-\sigma_2(B_3)}+1)}
\end{align*}
with $\sigma_1(B_i)$, $\sigma_2(B_i)$ and $\mu_i$ as above.  
\end{proposition}

\begin{proof} As usual, let $\mu_i>0$ be the middle eigenvalue of $\rho_\pi(A_i)$. Then
$$ 
\mu_i = \frac{\tau_i  - \sqrt{\tau_i^2 - \frac4{ \lambda_i}}}{2}
$$
because $\tau_i=\mu_i + \nu_i$ with $0<\lambda_i < \mu_i < \nu_i$ and $\lambda_i  \mu_i \nu_i=1$.

As in the proof of Proposition~\ref{prop:GoldmanBoudaryFromFockGoncharov}, 
\begin{align*}
\log \lambda_i & = {\textstyle \frac13} \sigma_1(B_{i+1}) +  {\textstyle \frac23} \sigma_2(B_{i+1}) + {\textstyle \frac23} \sigma_1(B_{i-1}) + {\textstyle \frac13} \sigma_2(B_{i-1}) +{\textstyle \frac23}\bigl(\tau_{111}(T_+) +\tau_{111}(T_-)\bigr),
\\
\log \mu_i &=  {\textstyle \frac13} \sigma_1(B_{i+1}) -  {\textstyle \frac13}\sigma_2(B_{i+1}) - {\textstyle \frac13} \sigma_1(B_{i-1}) +  {\textstyle \frac13}\sigma_2(B_{i-1}) -{\textstyle \frac13}\bigl(\tau_{111}(T_+) +\tau_{111}(T_-)\bigr) ,
\end{align*}
while Proposition~\ref{prop:InternalParam} gives
$$
\log s = {\textstyle \frac16}\sigma_1(B_1) + {\textstyle \frac16}\sigma_1(B_2) + {\textstyle \frac16}\sigma_1(B_3)-  {\textstyle \frac16}\sigma_2(B_1) - {\textstyle \frac16}\sigma_2(B_2) - {\textstyle \frac16}\sigma_2(B_3) .
$$
Solving this system of seven linear equations in seven unknown, we obtain
\begin{align*}
\sigma_1(B_i)&= {\textstyle \frac12} \log\lambda_{i-1} -  {\textstyle \frac12} \log\lambda_{i} +  {\textstyle \frac12} \log\lambda_{i+1} + \log \mu_{i-1} + \log s
\\
\sigma_2(B_i)&= {\textstyle \frac12} \log\lambda_{i-1} -  {\textstyle \frac12} \log\lambda_{i} +  {\textstyle \frac12} \log\lambda_{i+1} + \log \mu_{i+1} - \log s
\\
\tau_{111}(T_+) + \tau_{111}(T_-) &= -\log \mu_1 -\log \mu_2 -\log \mu_3.
\end{align*}

This provides the expressions of $\sigma_1(B_i)$ and $\sigma_2(B_i)$ indicated in the statement of Proposition~\ref{prop:FromGoldmanToFockGoncharov}. The triangle invariant $\tau_{111}(T_+)$ is obtained from the expression of the internal parameter $t$ in Proposition~\ref{prop:InternalParam}, from which we then deduce $\tau_{111}(T_-)$ by the above computation of $\tau_{111}(T_+)+\tau_{111}(T_-)$. 
\end{proof}
	
\section{More general surfaces}	
\label{sect:GeneralSurfaces}

We now consider the case of a compact oriented surface $S$ of genus $g$ with $n$ boundary components. This includes closed surfaces, where $n=0$. Goldman actually allows non-orientable surfaces in \cite{Goldman}, but the definitions of \cite{BonDre1} make heavy use of an orientation. 

In \cite{Goldman}, Goldman shows that the space $\mathfrak P(S)$ of isotopy classes of convex projective structures on $S$ is diffeomorphic to $\R^{16g+8n-16}$, by constructing explicit coordinates for this space. For this, he uses a \emph{pair of pants decomposition} of the surface, namely a family $C$ of disjoint simple closed curves in the interior of $S$  such that each component of $S-C$ is a pair of pants. An Euler characteristic computation shows that $C$ has $ 3g+n-3$ components, and that $S-C$ consists of $2g+n-2$ pairs of pants.  Choose an arbitrary orientation on each component $C_i$ of $C \cup \partial S$, with $i=1$, 2, \dots, $ 3g+2n-3$. 

Given this topological data, Goldman's coordinates are given by various invariants associated to each convex projective structure $\pi \in \mathfrak P(S)$ with monodromy $\rho_\pi \colon \pi_1(S) \to \SL$. 

The first set of coordinates are the numbers $\lambda_i$ and $\tau_i = \mu_i + \nu_i$ associated to the eigenvalues $0<\lambda_i <\mu_i <\nu_i$ of a matrix of $\SL$ representing $\rho_\pi(C_i) \in \SL$, as in \S \ref{sect:Goldman}. These invariants are constrained by the inequalities
$$2\lambda_i^{-\frac12} < \tau_i < \lambda_i + \lambda_i^{-2}.$$ 
This gives $6g+4n-6$ invariants.

Then, for each component $P_j$ of $S-C$, we have the two internal parameters $s_j$, $t_j>0$ constructed in \S \ref{sect:Goldman}. Note that, when a curve $C_i$ is a boundary component of a pair of pants $P_j$, this construction may require reversing the orientation of $C_i$ to make it opposite the boundary orientation, in order to match the conventions of  \S \ref{sect:Goldman}; this amounts to replacing $\lambda_i$ by $\lambda_i' = \nu_i^{-1} =\frac2{\tau_i  + \sqrt{\tau_i^2 - \frac4{ \lambda_i}}}$ and $\tau_i$ by $\tau_i'= \lambda_i^{-1} + \mu_i^{-1}= \lambda_i^{-1} + \frac2{\tau_i  - \sqrt{\tau_i^2 - \frac4{ \lambda_i}}}$. 

This gives $4g+2n-4$ additional invariants. 

Finally, Goldman identifies two additional degrees of freedom for each curve $C_i$ of $C$. The corresponding invariants $(u_i, v_i) \in \R^2$ are only defined up to a translation in $\R^2$. This ambiguity is very analogous to the well-known difficulty in defining the Fenchel-Nielsen twist coordinates for the Teichm\"uller space $\mathcal T(S)$. 

More precisely, the curve $C_i$ defines an action of $\R^2$ that generalizes the earthquake flow along $C_i$ on the Teichm\"uller space $\mathcal T(S)$.

The $\R\times 0$ part of the action is defined by the \emph{twist deformation} $T_{C_i}^u \colon \mathfrak P(S) \to \mathfrak P(S)$ which, for each $u\in \R$, modifies the projective structure $\pi \in \mathfrak P(S)$ on the right-hand  side of the curve $C_i$ by composing its charts with the projective map $\RP \to \RP$ induced by the linear map $\R^3 \to \R^3$ with matrix 
$$ 
\begin{pmatrix}
\mathrm e^{-u} & 0 &0\\
0& 1& 0 \\
0&0& \mathrm e^{u}
\end{pmatrix}
$$
in a basis where the coordinate vectors are eigenspaces of $\rho_\pi(C_i) \in \SL$ respectively corresponding to the eigenvalues $\lambda_i$, $\mu_i$, $\nu_i$ (with the usual convention that $0< \lambda_i < \mu_i < \nu_i$). As in the classical definition of the Fenchel-Nielsen twists in 2--dimensional hyperbolic geometry, the precise construction of this deformation requires us to work in the universal cover $\widetilde S$ and to define the projective structure $T_{C_i}^u(\pi)$ by deformation of the developing map $\mathrm{dev}_\pi \colon \widetilde S \to \RP$; the reader should  have no problem reconstructing the details of this construction (and can always consult \cite{Goldman}).

The  \emph{bulge deformation} $B_{C_i}^v \colon \mathfrak P(S) \to \mathfrak P(S)$ is similarly defined by using the matrix
$$ 
\begin{pmatrix}
\mathrm e^{-v} & 0 &0\\
0& \mathrm e^{2v} & 0 \\
0&0& \mathrm e^{-v}
\end{pmatrix}
$$
instead. 

Goldman's coordinates $(u_i, v_i)\in \R^2$ are only defined up to translation of $\R^2$. Their main property is that, if  $\pi \in \mathfrak P(S)$ has coordinates $(u_i, v_i)$, then its twist deformation $T_{C_i}^u(\pi)$ has coordinates $(u_i+u, v_i)$ while its bulge deformation $B_{C_i}^v(\pi)$ has coordinates $(u_i, v_i+v)$.

Since the pair of pant decomposition $C$ has $ 3g+n-3$ components, this provides $6g+2n-6$ additional coordinates. With the eigenvalue coordinates $(\lambda_i, \mu_i)$ associated to the $ 3g+2n-3$ components of $C \cup \partial S$ and the internal parameters $(s_i, t_i)$ associated to the $2g+n-2$ components of $S-C$, we now have a total of 
$$
2(3g+n-3) + 2(3g+2n-3) + 2(2g+n-2) = 16g +8n -10 = -8 \chi(S)
$$
coordinates, constrained by the inequalities that $0<2\lambda_i^{-\frac12} < \tau_i < \lambda_i + \lambda_i^{-2}$, $s_j>0$ and $t_j>0$.

Goldman shows that these coordinates provide a diffeomorphism between $\mathfrak P(S)$ and the polytope in $\R^{ -8 \chi(S)}$ defined by these constraints. 

\medskip

In \cite{BonDre1}, Bonahon and Dreyer introduce a different set of coordinates for $\mathfrak P(S)$,  similarly associated to a pair of pants decomposition $C$. They consider the geodesic lamination $\Lambda$ that is the union of $C\cup \partial S$ and of three spiraling lines as in Figure~\ref{fig:PantsSpiral} for each pair of pants component of $S-C$, together with additional  topological information at each component $C_i$ of $C$ described by a small arc transverse to $\Lambda$ and cutting $C_i$ in exactly one point. In particular, the complement $S-\Lambda$ consists of $4g+2n-4$ infinite triangles, two for each component of $S-C$. 

Given a projective structure $\pi \in \mathfrak P(S)$, they associate an invariant $\tau_{111}(T)\in \R$ to each component $T$ of $S-\Lambda$, two shear invariants $\sigma_1(B)$, $\sigma_2(B)\in \R$ to each spiraling leaf $B$ of $\Lambda$, and two more shear invariants $\sigma_1(C_i)$, $\sigma_2(C_i)\in \R$ to each component $C_i$ of the pair of pants decomposition $C$. 

The coordinates $\tau_{111}(T)$, $\sigma_1(B)$, $\sigma_2(B)$ are the Fock-Goncharov coordinates arising in the pair of pants of $S-C$, and \S \ref{sect:PantsFromFGtoGoldman} and \S \ref{sect:PantsFromGoldmanToFG} indicate how these are connected to the Goldman coordinates. 

The shear invariants $\sigma_1(C_i)$, $\sigma_2(C_i)\in \R$ of $\pi \in \mathfrak P(S)$ are well-defined by the topological data (not just up to translation of $\R^2$). More precisely, lift $C_i$ to a line $\widetilde C_i$ in the universal cover $\widetilde S$ that is invariant under $C_i \in \pi_1(S)$. The topological data provided by the little arc transverse to $C_i$ then determines a pair $(\widetilde B_i^{\mathrm{left}}, \widetilde B_i^{\mathrm{right}})$, uniquely defined modulo the action of   $C_i \in \pi_1(S)$, leaves of the preimage in $\widetilde S$ of the geodesic lamination $\Lambda$; such that: 
\begin{itemize}
\item $\widetilde B_i^{\mathrm{left}}$ and $\widetilde B_i^{\mathrm{right}}$ each lift a spiraling leaf of $\Lambda$;
\item $\widetilde B_i^{\mathrm{left}}$ is asymptotic to the positive endpoint of $\widetilde C_i$, and is located on the left-hand side of $\widetilde C_i$;
\item $\widetilde B_i^{\mathrm{right}}$ is asymptotic to the negative endpoint of $\widetilde C_i$, and is located on the right-hand side of $\widetilde C_i$.
\end{itemize}
See \cite{BonDre1} for details. 
As in \S \ref{sect:FockGoncharov}, the developing map $\mathrm{dev}_\pi \colon \widetilde S \to \RP$ and its monodromy $\rho_\pi \colon \pi_1(S) \to \SL$  associate flags $E_i$, $F_i$, $G_i$, $H_i$ to the positive endpoint of $\widetilde C_i$, the negative endpoint of $\widetilde C_i$, the endpoint of $\widetilde B_i^{\mathrm{left}}$ that is not an endpoint of $\widetilde C_i$, and the endpoint of $\widetilde B_i^{\mathrm{right}}$ that is not an endpoint of $\widetilde C_i$, respectively. 
The invariants $\sigma_1(C_i)$, $\sigma_2(C_i)\in \R$ of $\pi \in \mathfrak P(S)$  are then defined by the property that
$$
\sigma_1(C_i) = \log \left(-
\frac{e_{i}^{(1)} \wedge f_{i}^{(1)} \wedge g_i^{(1)}}
{e_{i}^{(1)} \wedge f_{i}^{(1)} \wedge h_i^{(1)}}
\frac{ f_{i}^{(2)} \wedge h_i^{(1)}}
{f_{i}^{(2)} \wedge g_i^{(1)}}
\right)
$$
and
$$
\sigma_2(C_i) =  \log \left(-
\frac{e_{i}^{(1)} \wedge f_{i}^{(1)} \wedge h_i^{(1)}}
{e_{i}^{(1)} \wedge f_{i}^{(1)} \wedge g_i^{(1)}}
\frac{ e_{i}^{(2)} \wedge g_i^{(1)}}
{e_{i}^{(2)} \wedge h_i^{(1)}}
\right)
$$
with our usual conventions that $e_i^{(a)} \in \Lambda^a(E_i^{(a)})$, $f_i^{(a)} \in \Lambda^a(F_i^{(a)})$, $g_i^{(a)} \in \Lambda^a(G_i^{(a)})$ and $h_i^{(a)} \in \Lambda^a(H_i^{(a)})$.

\begin{lemma}
\label{lem:TwistBulge}
Let $C_i$ be a component of the  pair of pants decomposition $C$. If the corresponding shear invariants of $\pi \in \mathfrak P(S)$ are $\sigma_1(C_i)$ and $\sigma_2(C_i)$, then the projective structure $T_{C_i}^u(\pi) \in \mathfrak P(S)$ defined by twisting  $\pi$ along $C_i$ has shear invariants $\sigma_1(C_i)+u$ and $\sigma_2(C_i)+u$. Similarly,  the projective structure $B_{C_i}^v(\pi)\in \mathfrak P(S)$ defined by bulging $\pi$ along $C_i$ has shear invariants $\sigma_1(C_i)-3v$ and $\sigma_2(C_i)+3v$.
\end{lemma}

\begin{proof} 
By definition of the stable and unstable flags of $\rho_\pi(C_i) \in \SL$, this matrix  sends $e_i^{(1)}$ to $\nu_i e_i^{(1)}$ and $f_i^{(1)}$ to $\lambda_i f_i^{(1)}$. Let $k_i \in \R^3$ be an eigenvector corresponding to the second eigenvalue $\mu_i$ of $\rho_\pi(C_i)$, with $0<\lambda_i <\mu_i <\nu_i$. In particular, in the above formulas for $\sigma_1(C_i)$ and $\sigma_2(C_i)$, we can take $e_i^{(2)} = e_i^{(1)}\wedge k_i\in \Lambda^2(E^{(2)})$ and $f_i^{(2)} = f_i^{(1)}\wedge k_i\in \Lambda^2(F^{(2)})$. 

By construction, the twist deformation $T_{C_i}^u$ replaces $h_i^{(1)}$ by $T(h_i^{(1)})$, where $T\colon \R^3 \to \R^3$ sends $e_i^{(1)}$ to $\mathrm e^{u} e_i^{(1)}$, $k_i$ to $k_i$, and $f_i^{(1)}$ to $\mathrm e^{-u} f_i^{(1)}$. By consideration of the impact of this transformation on the above formulas for  $\sigma_1(C_i)$ and $\sigma_2(C_i)$, the shear invariants of  $T_{C_i}^u(\pi)$ are respectively $\sigma_1(C_i)+u$ and $\sigma_2(C_i)+u$. 

Similarly, the bulge deformation $B_{C_i}^u$ replaces $h_i^{(1)}$ by $B(h_i^{(1)})$, where $B\colon \R^3 \to \R^3$ sends $e_i^{(1)}$ to $\mathrm e^{-v} e_i^{(1)}$, $k_i$ to $\mathrm e^{2v} k_i$, and $f_i^{(1)}$ to $\mathrm e^{-v} f_i^{(1)}$. It follows that the shear invariants of  $B_{C_i}^u(\pi)$ are respectively $\sigma_1(C_i)-3v$ and $\sigma_2(C_i)+3v$. 
\end{proof}

\begin{proposition}
\label{prop:TwistBulge}
Up to translation by a vector $(u_i^0, v_i^0)\in \R^2$  depending on the choices made in Goldman's construction, the Goldman parameters $(u_i, v_i)$ associated to the closed curve $C_i$ are related to the Bonahon-Dreyer shear invariants $\sigma_1(C_i)$ and $\sigma_2(C_i)$ by the property that
\begin{align*}
\sigma_1(C_i) &= (u_i-u_i^0) - 3(v_i -v_i^0)\\
\text{and }\sigma_2(C_i) &= (u_i-u_i^0) + 3(v_i -v_i^0).
\end{align*}
Similarly,
\begin{align*}
u_i &= {\textstyle \frac12} \big(\sigma_1(C_i) + \sigma_2(C_i) \big) + u_i^0\\
\text{and }v_i &= {\textstyle \frac16} \big(-\sigma_1(C_i) + \sigma_2(C_i) \big) + v_i^0. 
\end{align*}
\end{proposition}
\begin{proof}
The first set of equations is an immediate consequence of Lemma~\ref{lem:TwistBulge}. A straightforward computation then gives the second set.
\end{proof}

This completes our description of the correspondence between the coordinates of \cite{Goldman} for $\mathfrak P(S)$ and those of \cite{BonDre1}.

\end{document}